\documentclass[12pt]{article}
\usepackage{graphicx}
\usepackage{amssymb}
\usepackage{amsmath}
\usepackage{amsfonts}
\usepackage{amsthm}
\usepackage{accents}
\usepackage[english]{babel}
\usepackage[latin1,ansinew]{inputenc}
\usepackage{a4wide}
\usepackage{tikz}
\usepackage{color}
\newcommand{\be}{\begin{eqnarray}}
\newcommand{\ee}{\end{eqnarray}}
\newtheorem{theo}{Theorem}
\newtheorem{prop}{Proposition}

\newtheorem{lemma}{Lemma}

\newcommand{\R}{\mathbb R}
\newcommand{\C}{\mathbb C}

\newcommand{\bxi}{\boldsymbol\xi}
\newcommand{\bomega}{\boldsymbol\omega}

\newcommand{\mJ}{{\mathcal J}}

\begin{document}
\date{}

\title{Uniform Decay for \\  
Ruggeri's Relaxed Navier-Stokes-Fourier Equations}
\author{Heinrich Freist\"uhler\thanks{Department of Mathematics, University of Konstanz, 
78457 Konstanz, Germany}\qquad Yizhou Zhou\thanks{Corresponding author,
IGPM, RWTH Aachen, 52072 Aachen, Germany (zhou@igpm.rwth-aachen.de)} 
}
\maketitle

\begin{abstract}
This work studies the linearization of Ruggeri's classical symmetric hyperbolic model of viscous, heat-conductive fluids about a homogeneous equilibrium. In three space dimensions, we show that perturbations in \(L^1\cap L^2\) decay in \(L^2\) at the rate \((1+t)^{-3/4}\), uniformly as a common parameter scaling the three relaxation times tends to zero. The estimate controls the equilibrium variables together with the appropriately weighted stresses and heat flux, and has the same decay rate as the formal Navier--Stokes--Fourier limit.
\end{abstract}

\section{The result}
In his landmark 1983 paper \cite{R83}, the starting point of Rational Extended Thermodynamics
\cite{MR,RS}, Ruggeri gave the first formulation for the dynamics of viscous, heat conductive   
compressible fluids as a symmetric hyperbolic system of conservation laws. His equations of 
motion, in which in the list of independent variables the equilibrium quantities, temperature 
$\theta$, thermodynamic potential $\psi$, and velocity $u$, are complemented by the stress 
$\Sigma+\sigma I$ and the heat flux $q$, 
are a physically well justified relaxed version of the Navier-Stokes-Fourier system,
in which three Maxwell-Cattaneo type retardation mechanisms implement finite speed of propagation
for the dissipation effects of shear viscosity, bulk viscosity, and heat conduction.  

The purpose of the present work is to study the uniform decay property of Ruggeri's model, as the relaxation parameter tends to zero. 
As a prerequisite for a later study of the fully nonlinear situation, we here investigate the linear case. 
According to \cite{F22}, the linearization of Ruggeri's model 
at a homogeneous equilibrium state 
$(\bar\theta,\bar\psi,\bar u,\bar\Sigma,\bar\sigma,\bar q)=(\bar\theta,\bar\psi,0,0,0,0)$ is given by
\be\label{linR}
A^0 U_t+\sum_{j=1}^3A^jU_{x_j}=B U, \quad U=(V,W)=((\check\psi,u,\check\theta),(\Sigma,\sigma,q))
\ee
with  
$$
A^0_\epsilon=
\begin{pmatrix}
       \theta^{-1}\hat p_{\psi\psi}&0&-\hat p_\psi+\theta \hat p_{\theta\psi}&0&0&0\\
       0&\hat p_\psi&0&0&0&0\\
       -\hat p_\psi+\theta \hat p_{\theta\psi}&0&\theta^3\hat p_{\theta\theta}&0&0&0\\
 0&0&0&\epsilon^2\tau_1\rho\delta^{(kl)(rs)}&0&0\\
   0&0&0&0&\epsilon^2\tau_2\rho&0\\
   0&0&0&0&0&\epsilon^2\tau_3\rho\delta^{mn}
       \end{pmatrix}
$$
$$
\ A^j=
\begin{pmatrix}
       0&\hat p_\psi\delta^{nj}&0&0&0&0\\
       \hat p_\psi\delta^{mj}&0&\theta^2 \hat p_\theta \delta^{mj}&\theta C^{mjrs}&\theta\delta^{mj}&0\\
       0&\theta^2\hat p_\theta\delta^{nj}&0&0&0&\theta^2\delta^{nj}\\ 
       0&\theta C^{njkl}&0&0&0&0\\
       0&\theta \delta^{nj}&0&0&0&0\\
       0&0&\theta^2\delta^{mj}&0&0&0      
       \end{pmatrix},\quad j=1,2,3,
$$
$$
B=
\begin{pmatrix}
   \ 0\ &\ 0\ &\ 0\ &0&0&0\\
   0&0&0&0&0&0\\
   0&0&0&0&0&0\\
   0&0&0&-(1/\eta)\delta^{(kl)(rs)}&0&0\\
   0&0&0&0&-1/\zeta&0\\
   0&0&0&0&0&-(1/\chi)\delta^{mn}
\end{pmatrix},\qquad\qquad
$$
where $p=\hat p(\theta,\psi)$, its derivatives,  and $\rho=\hat p_\psi(\theta,\psi)/\theta$ are evaluated 
at the reference state $(\bar\theta,\bar\psi)$.
We assume throughout that $\bar\theta>0$, $\eta,\zeta,\chi>0$,
$\hat p_\psi(\bar\theta,\bar\psi)>0$, and that the $2\times2$ matrix
$$
\begin{pmatrix}
\bar\theta^{-1}\hat p_{\psi\psi}
&
-\hat p_\psi+\bar\theta\hat p_{\theta\psi}\\[1mm]
-\hat p_\psi+\bar\theta\hat p_{\theta\psi}
&
\bar\theta^3\hat p_{\theta\theta}
\end{pmatrix}.
$$
is positive definite.
These are the equilibrium versions of the standard conditions
\cite{F22}. The coefficients $\tau_1,\tau_2,\tau_3$ are
positive, and $\Sigma$ is symmetric and trace-free.

The goal of this work is to show the following $L^2$ estimate, which is the starting point for its obvious extensions
to $L^2$-based Sobolev spaces.
\begin{theo}
There exists a constant $C>0$ such that for
any $\epsilon\in(0,\epsilon_0]$ with $\epsilon_0$ a sufficiently small constant, and any data 
$$
U_0=(V_0,W_0)=((\check\psi_0,u_0,\check\theta_0),(\Sigma_0,\sigma_0,q_0))\in L^2\cap L^1,
$$
the solution $t\mapsto (V(t),W(t))$ of system \eqref{linR} with these data obeys the estimate 
$$
\Vert(V,\epsilon W)\Vert_{L^2}
\le C(1+t)^{-3/4}(\Vert(V_0,\epsilon W_0)\Vert_{L^2}+\Vert(V_0,\epsilon W_0)\Vert_{L^1})\quad\text{for all }t>0.
$$
\end{theo}

The significance of Theorem 1 lies in the uniformity of the decay estimate with respect to the relaxation parameter. For each fixed \(\epsilon>0\), large-time decay may be obtained from the standard dissipative theory for hyperbolic relaxation systems, see \cite{SK85,F22}; however, the corresponding constants generally need not remain bounded as \(\epsilon\searrow0\). Theorem 1 shows that the solutions decay at the three-dimensional diffusion rate \((1+t)^{-3/4}\), uniformly for sufficiently small \(\epsilon\). In particular, this estimate exhibits the same large-time decay rate as the formal limiting Navier--Stokes--Fourier system.
\be\label{NSV}
A_{01}\bar V_t+\sum_{j=1}^dA_{j11}\bar V_{x_j}=
\sum_{i,j=1}^d A_{i12}J^{-1} A_{j21}\bar V_{x_ix_j}
\ee
However, Theorem 1 is a uniform decay result rather than a relaxation-limit theorem; quantitative convergence as \(\epsilon\searrow0\) requires a separate comparison argument.



The hyperbolic relaxation of the compressible Navier--Stokes--Fourier equations and its zero-relaxation limit have also been studied extensively. A pioneering contribution in this direction is due to Yong \cite{Y}, 
in which the author, independently of \cite{R83}, revised Maxwell's idea of relaxing
the viscous stress by doing it separately for the shear and the bulk viscosity, in the nonlinear setting. His system has a conservation-dissipation structure and is symmetrizable hyperbolic.
For that setting Yong showed 
convergence, as $\epsilon\searrow 0$, on every fixed interval $[0,T]$ on which the limiting compressible Navier-Stokes solution remains smooth.

This line of research was subsequently continued by Hu and Racke \cite{HR1,HR2,HR3}. The multidimensional non-isentropic case was studied in \cite{HR1}, where the shear and bulk stresses are relaxed separately as in Yong's model, while Fourier's law is retained for heat conduction. The analysis was extended in \cite{HR2} to a one-dimensional, fully hyperbolic model incorporating Galilean-invariant Maxwell and Cattaneo laws for the stress and heat flux. A multidimensional counterpart was considered in \cite{HR3}, although only the heat flux and bulk stress are relaxed, while the shear stress remains Newtonian. These works establish small-data global well-posedness and the corresponding zero-relaxation limits in finite time. In addition, \cite{HR3} proves finite-time blow-up for a class of large data when the shear viscosity vanishes. The models in \cite{HR2,HR3} possess a dissipative physical entropy; but since they
are not in divergence form, this entropy cannot be used in the sense of \cite{G,FL,B} to symmetrize the system.


Using Oldroyd-type constitutive laws, Peng and Zhao \cite{PZ1,PZ2} studied a first-order approximation of the full compressible Navier--Stokes equations. In the one-dimensional setting considered there, the resulting system coincides with that of \cite{HR2}. Uniform global existence and global-in-time convergence to the compressible Navier--Stokes equations were addressed. This result was further strengthened in \cite{PZ2}, where quantitative global-in-time convergence rates were obtained.
More refined ways of relaxing the dissipative mechanisms include upper convected Maxwell 
derivatives. Boyaval \cite{By} has considered this with respect to the viscosity. In 
\cite{CKX} Crin-Barat, Kawashima, and Xu relax the heat conduction in this way. They established uniform global well-posedness for small initial data, large-time decay estimates, and quantitative convergence to the Navier--Stokes--Fourier equations as the relaxation parameter tends to zero.

We now briefly comment on the proof of Theorem 1, which will follow in Sec.\ 2.
After taking the Fourier transform in space, the proof reduces to constructing a uniformly bounded change of basis under which the Hermitian part of the transformed symbol is bounded below by
\[
c\frac{|\xi|^2}{1+|\xi|^2},
\]
with \(c>0\) independent of sufficiently small \(\epsilon\). To construct this transformation, we divide the frequency space into three regimes. For small frequencies, perturbative block diagonalization separates the equilibrium and dissipative modes, and the Shizuta--Kawashima condition yields dissipation of order \(|\xi|^2\) on the equilibrium block. For intermediate frequencies, the desired estimate follows from compactness and the strict positivity of the real parts of the relevant eigenvalues. The large-frequency regime is more involved because the two parameters \(\epsilon\) and \(|\xi|^{-1}\) may approach zero at different rates. We therefore introduce a polar-type parametrization and divide the parameter space according to their relative sizes. A separate perturbative block diagonalization in each region then gives the required uniform estimate. The theorem finally follows from the resulting pointwise Fourier estimate and the standard low--high frequency decomposition.


\section{The proof} 
For the linearized Ruggeri's model, we write 
$$
A^0_\epsilon=\begin{pmatrix}
       A^0_1&0\\[1mm]
       0&\epsilon^2 A^0_2 
       \end{pmatrix},
\quad
\sum_{j=1}^3\xi_jA^j=\begin{pmatrix}
            A_{11}(\bxi)&A_{12}(\bxi)\\[1mm]
            A_{21}(\bxi)&0  
            \end{pmatrix},
\quad
B=\begin{pmatrix}
   0&0\\[1mm]
   0&-\mJ 
   \end{pmatrix}
$$
and consider, with $\bxi=(\xi_1,\xi_2,\xi_3)=\xi\bomega,\bomega\in S^2,\xi\ge 0$, 
$$
M(\xi,\epsilon,\bomega)=(A^0_\epsilon)^{-1/2}\left(i\sum_{j=1}^3\xi_jA^j-B\right)(A^0_\epsilon)^{-1/2}
$$
i.e.,
\begin{align*}
M(\xi,\epsilon,\bomega)=&
\begin{pmatrix}
\phantom{\epsilon^{-1/2}}i\xi (A_1^0)^{-1/2}A_{11}(\bomega)(A_1^0)^{-1/2}& \ i\epsilon^{-1}\xi (A_1^0)^{-1/2}A_{12}(\bomega)(A_2^0)^{-1/2}\\[2mm]
i\epsilon^{-1}\xi (A_2^0)^{-1/2}A_{21}(\bomega)(A_1^0)^{-1/2}&\quad\ \ \epsilon^{-2}(A_2^0)^{-1/2}~\mJ~(A_2^0)^{-1/2}
\end{pmatrix}\\[2mm]
=&
\begin{pmatrix}
\quad\ \ \ i\xi \tilde A_{11}(\bomega)&\ i\epsilon^{-1}\xi\tilde A_{12}(\bomega)\\[2mm]
i\epsilon^{-1}\xi\tilde A_{21}(\bomega)&\epsilon^{-2}\tilde\mJ
\end{pmatrix},
\end{align*}
with 
$$
\begin{aligned}
\tilde A_{11}(\bomega)=(A^0_1)^{-1/2}A_{11}(\bomega)(A^0_1)^{-1/2},\quad
\tilde A_{12}(\bomega)&=(A^0_1)^{-1/2}A_{12}(\bomega)(A^0_2)^{-1/2},\\[2mm]
\tilde A_{21}(\bomega)=(A^0_2)^{-1/2}A_{21}(\bomega)(A^0_1)^{-1/2},\qquad\quad\
\tilde \mJ&=(A^0_2)^{-1/2}~\mJ~ (A^0_2)^{-1/2}.
\end{aligned}
$$

Following the Fourier argument in \cite{K83,UKS}, Theorem 1 is a 
consequence of the following proposition:
\begin{prop}
There exists $c>0$ and uniform transformations $R(\xi,\epsilon,\bomega)$ such that for all 
$\bomega\in S^{d-1},\epsilon\in(0,\epsilon_0],\xi\ge 0$, the transformed version 
$$
\tilde M(\xi,\epsilon,\bomega)=R(\xi,\epsilon,\bomega)^{-1}M(\xi,\epsilon,\bomega)R(\xi,\epsilon,\bomega)
$$
satisfies
$$
(*)\quad\tilde M(\xi,\epsilon,\bomega)+\tilde M(\xi,\epsilon,\bomega)^*\ge c\frac{|\bxi|^2}{1+|\bxi|^2}.
$$
\end{prop}

Indeed, let
\[
\widehat U(\bxi,t)
=
\begin{pmatrix}
\widehat V(\bxi,t)\\[1mm]
\widehat W(\bxi,t)
\end{pmatrix},
\qquad
Z(\bxi,t)
=
(A_\epsilon^0)^{1/2}\widehat U(\bxi,t).
\]
Taking the Fourier transform of \eqref{linR} with respect to $x$, we obtain
\[
Z_t+M(\bxi,\epsilon)Z=0.
\]
Suppose that there exists a uniformly regular transformation
$R=R(\bxi,\epsilon)$, meaning that $R$ and $R^{-1}$ are bounded
uniformly in $\bxi$ and $\epsilon$, such that
\[
\tilde M=R^{-1}MR
\]
satisfies
\[
\tilde M+\tilde M^*
\geq
c\frac{|\bxi|^2}{1+|\bxi|^2}I.
\]
Setting $\tilde Z=R^{-1}Z$, we find
\begin{align*}
\frac{d}{dt}|\tilde Z|^2
&=
-\tilde Z^*
\bigl(\tilde M+\tilde M^*\bigr)
\tilde Z \leq
-c\frac{|\bxi|^2}{1+|\bxi|^2}
|\tilde Z|^2.
\end{align*}
The uniform boundedness of $R$ and $R^{-1}$ therefore gives
\[
|Z(\bxi,t)|
\leq
C\exp\left(
-c\frac{|\bxi|^2}{1+|\bxi|^2}t
\right)
|Z(\bxi,0)|.
\]
Since $A_{01}$ and $A_{02}$ are positive definite, $|Z|$ is
uniformly equivalent to
\(|(\widehat V,\epsilon\widehat W)|\).
We now divide the Fourier space into the low-frequency region
$|\bxi|\leq1$ and the high-frequency region $|\bxi|\geq1$. For the
low-frequency part, the preceding estimate and
\[
\|(\widehat V_0,\epsilon\widehat W_0)\|_{L^\infty}
\leq
\|(V_0,\epsilon W_0)\|_{L^1}
\]
yield, in dimension $d=3$,
\begin{align*}
\int_{|\bxi|\leq1}
|(\widehat V,\epsilon\widehat W)(\bxi,t)|^2\,d\bxi &\leq
C\|(V_0,\epsilon W_0)\|_{L^1}^2
\int_{|\bxi|\leq1}e^{-ct|\bxi|^2}\,d\bxi\\
& \leq
C(1+t)^{-3/2}
\|(V_0,\epsilon W_0)\|_{L^1}^2.
\end{align*}
For the high-frequency part, we have
\[
\int_{|\bxi|\geq1}
|(\widehat V,\epsilon\widehat W)(\bxi,t)|^2\,d\bxi
\leq
Ce^{-ct}
\|(V_0,\epsilon W_0)\|_{L^2}^2.
\]
Combining the two estimates and using Plancherel's theorem, we obtain
\[
\|(V,\epsilon W)(t)\|_{L^2}
\leq
C(1+t)^{-3/4}
\left(
\|(V_0,\epsilon W_0)\|_{L^2}
+
\|(V_0,\epsilon W_0)\|_{L^1}
\right).
\]
This proves Theorem 1. It therefore remains to establish the stated
uniform dissipativity property of the Fourier symbol $M$. We do that separately in three regimes. 

\subsection{Small $\xi$}
\par\smallskip
We first consider the case where $\xi\leq \underline k$ for a sufficiently small constant $\underline k>0$. In this case, the result follows from the following proposition.
\begin{prop}\label{prop2-small}
There exist a constant $c > 0$ and a uniformly regular base transform $T = T(\xi, \epsilon)$ 
such that for
$$
N(\xi,\epsilon)=\epsilon^2M(\xi,\epsilon)=
\begin{pmatrix}
i\epsilon^2\xi\tilde A_{11}(\bomega)&\ i\epsilon\xi\tilde A_{12}(\bomega)\\[1mm]
\ i\epsilon\xi\tilde A_{21}(\bomega)&\tilde\mJ    
\end{pmatrix},
$$
the matrix $\tilde N(\xi,\epsilon)=T^{-1}(\xi, \epsilon)N(\xi, \epsilon)T(\xi, \epsilon)$ satisfies 
$$
\tilde N(\xi, \epsilon) + \tilde N^*(\xi, \epsilon)\geq c\epsilon^2\xi^2I
$$
for sufficiently small $\epsilon$ and $\xi$.
\end{prop}
\begin{proof}[Proof of Proposition \ref{prop2-small}]
    Note that the unperturbed matrix 
    $$
    N(0,0)=\begin{pmatrix}
        0 & \\
          & \tilde \mJ
    \end{pmatrix}
    $$
    has $5$ zero eigenvalues and $9$ positive eigenvalues. Then according to the perturbation theory \cite{Kato-book}, there exists a regular transform $T_1(\xi,\epsilon)$ for sufficiently small $\xi$ and $\epsilon$ such that 
    $$
    T_1^{-1}(\xi, \epsilon)N(\xi, \epsilon)T_1(\xi, \epsilon) = \begin{pmatrix}
        V(\xi, \epsilon) & \\[1mm]
          & W(\xi, \epsilon)
    \end{pmatrix}
    $$
    with $V(\xi,\epsilon)\in \mathbb{C}^{5\times 5}$ and $W(\xi,\epsilon)\in \mathbb{C}^{9\times 9}$. Moreover,  the transform matrix has the form
    $$
    T_1(\xi, \epsilon) = 
    \begin{pmatrix}
        I  & X(\xi, \epsilon)\\[1mm]
        Y(\xi, \epsilon)  & I
    \end{pmatrix}
    $$
    with $X(\xi,\epsilon)\in \mathbb{C}^{5\times 9}$ and $Y(\xi,\epsilon)\in \mathbb{C}^{9\times 5}$.
    By the computation in Appendix, we have 
    $$
    V(0,0)=0,\qquad W(0,0)=\tilde\mJ.
    $$
    Notice that $W(0,0)$ satisfies $W(0,0)\geq cI$ with a constant $c>0$. It implies that 
    $$
    W(\xi,\epsilon)+W^*(\xi,\epsilon)\geq cI
    $$
    for sufficiently small $\xi$ and $\epsilon$.
    
    Next we focus on the computation of $V(\xi,\epsilon)$.
    We first fix
    $\xi>0$ and consider the computation in Appendix with the varying parameter $\epsilon$. In this case, we have 
    $$
    N(\xi,0) = \begin{pmatrix}
        0 & \\
          & \tilde \mJ
    \end{pmatrix}
    $$
    and therefore obtain $V(\xi,0)=0$. Moreover, the relation \eqref{Vprime} in Appendix shows that
    \begin{equation}\label{prop2-eqV1}
    \frac{\partial V}{\partial\epsilon}(\xi,0)=0.
    \end{equation}
    These two relations imply that 
    $$ V(\xi,\epsilon)=\epsilon^2 \hat{V}(\xi,\epsilon)$$
    with $\hat{V}(\xi,\epsilon)\in \mathbb{C}^{5\times 5}$ being regular with respect to $(\xi,\epsilon)$.
   
    Next we fix $\epsilon>0$ and consider the computation in Appendix with the varying parameter $\xi$. The relation \eqref{V} gives
    $$V(0,\epsilon) = 0.$$ Combining this with \eqref{prop2-eqV1}, we know 
    $$
    V(\xi,\epsilon) = \epsilon^2\xi\bar{V}(\xi,\epsilon) $$ 
    with a regular matrix $\bar{V}(\xi,\epsilon)\in \mathbb{C}^{5\times 5}$.
    Furthermore, we see from \eqref{Vprime}--\eqref{Vprimeprime} that 
    $$
    \frac{\partial V}{\partial\xi}(0,\epsilon)=i\epsilon^2\tilde A_{11},\qquad 
    \frac{\partial^2 V}{\partial\xi^2}(0,\epsilon)=2\epsilon^2\tilde A_{12}\tilde \mJ^{-1}\tilde A_{21},
    $$
    which imply
    $$
    \bar{V}(0,\epsilon)=i\tilde A_{11},\qquad \frac{\partial \bar{V}}{\partial\xi}(0,\epsilon) = \tilde A_{12}\tilde \mJ^{-1}\tilde A_{21}
    $$
    for any small $\epsilon>0$.
    From these two relations, we conclude that 
    $$
    \bar{V}(\xi,\epsilon) = i\tilde A_{11} + \xi \tilde A_{12}\tilde \mJ^{-1}\tilde A_{21} + \xi^2 \tilde V(\xi,\epsilon)
    $$
    with $\tilde V(\xi,\epsilon)$ another regular function of $(\xi,\epsilon)$. 
    
    We know from the Shizuta-Kawashima condition that any eigenvector of matrix $\tilde A_{11}$ is positive definite on the matrix $\tilde A_{12}\tilde \mJ^{-1}\tilde A_{21}$. Then by Lemma 5 in \cite{FS21}, we conclude that there exists a regular transform $T_2=T_2(\xi,\epsilon)\in \mathbb{C}^{5\times 5}$ such that $$
    \left[T_2^{-1}\bar{V}T_2 \right]^*+T_2^{-1}\bar{V}T_2 \geq c\xi I
    $$
    for sufficiently small $\xi$. At last, we conclude
    $$
    T(\xi,\epsilon) =  T_1(\xi,\epsilon)\begin{pmatrix}
        T_2(\xi,\epsilon) & 0\\[1mm]
        0 & I
        \end{pmatrix},\qquad 
        \tilde{N}(\xi,\epsilon)= 
        \begin{pmatrix}
        \epsilon^2\xi ~T_2^{-1}\bar{V}(\xi,\epsilon)T_2 & 0 \\[1mm]
        0 & W(\xi,\epsilon)
        \end{pmatrix}
    $$
    and 
    $$
    \tilde{N}(\xi,\epsilon)+\tilde{N}(\xi,\epsilon)^* \geq c\epsilon^2\xi^2 I.
    $$
    This completes the proof of proposition.
\end{proof}

\subsection{Intermediate $|\xi|$}
\par\smallskip
For $|\xi|$ in the compact interval $[\underline k,\overline k]$, it is equivalent to prove the following assertion.
\begin{prop}\label{prop-intermediate}
    For any $\xi\in [\underline k,~\overline k]$, there exist a $c > 0 $ and a uniformly regular
base transform $T = T(\xi, \epsilon)$ for 
$$
N(\xi,\epsilon)=
\begin{pmatrix}
i\epsilon^2\xi\tilde A_{11}(\bomega)&\ i\epsilon\xi\tilde A_{12}(\bomega)\\[1mm]
\ i\epsilon\xi\tilde A_{21}(\bomega)&\tilde\mJ    
\end{pmatrix},
$$
such that the matrix $\tilde N(\xi,\epsilon)=T^{-1}(\xi, \epsilon)N(\xi, \epsilon)T(\xi, \epsilon)$ satisfies 
$$
\tilde N(\xi, \epsilon) + \tilde N^*(\xi, \epsilon)\geq c\epsilon^2I
$$
for sufficiently small $\epsilon$.
\end{prop}

 Since $[\underline k,~\overline k]$ is a compact set, it suffices to prove the proposition for fixed $\xi$.
    At the point $\epsilon=0$, we have 
    $$
    N(\xi,0)=\begin{pmatrix}
        0 & \\
          & \tilde \mJ
    \end{pmatrix}.
    $$
    Thus there exists a regular transform $T_1=T_1(\epsilon)$ such that 
    $$
    T_1^{-1}(\epsilon)N(\xi, \epsilon)T_1(\epsilon) = \begin{pmatrix}
        V(\epsilon) & \\[1mm]
          & W(\epsilon)
    \end{pmatrix}\quad \text{with}\quad 
    T_1(\epsilon) = 
    \begin{pmatrix}
        I  & X(\epsilon)\\[1mm]
        Y(\epsilon)  & I
    \end{pmatrix}.
    $$
    We use the computation in Appendix with the parameter $\epsilon$. In this case, we have $$W(0)=\tilde \mJ>0,\qquad V(0)=V'(0)=0$$ and 
    $$
    V''(0)=2(i\xi \tilde A_{11} + \xi^2 \tilde A_{12}\tilde \mJ^{-1}\tilde A_{21} ).
    $$
    It suffices to show that there exists a transform $T_2\in\mathbb{C}^{5\times 5}$ such that 
    $$
    T_2^{-1}V''(0)T_2 + \Big(T_2^{-1}V''(0)T_2\Big)^*\geq cI.$$
    By the Shizuta-Kawashima condition, we know that the matrix $i\xi \tilde A_{11} + \xi^2 \tilde A_{12}\tilde \mJ^{-1}\tilde A_{21}$ has eigenvalues with strictly positive real parts. Consequently, Proposition \ref{prop-intermediate} follows directly from the following lemma:
\begin{lemma}\label{lemma2}
Let $P\subset \R^N$ be compact and $K:P\to M_n(\C)$ be continuous such that all eigenvalues of all $K(p)$, $p\in P$, 
have positive real part.
Then there exist a constant $c>0$ and a map $T:P\to GL(n,\C)$ such that $\{T(p):p\in P\}$ and 
$\{T(p)^{-1}:p\in P\}$ are bounded and
$$
\tilde K(p)=T(p)^{-1}K(p)T(p)
$$
satisfies 
$$
\tilde K(p)+\tilde K(p)^*\ge cI\quad \text{for all } p\in P.
$$
\end{lemma}


\subsection{Large $|\xi|$}
\par\smallskip
For $\xi\geq \overline k$, we take $\eta=1/\xi$ and consider the matrix 
$$
H(\eta,\epsilon)=\epsilon^2\eta M(\xi,\epsilon)=
\begin{pmatrix}
i\epsilon^2\tilde A_{11}(\bomega)&\ i\epsilon\tilde A_{12}(\bomega)\\[1mm]
\ i\epsilon\tilde A_{21}(\bomega)&\eta \tilde\mJ    
\end{pmatrix}.
$$
The inequality $(*)$ in Proposition 1 follows from the following result.
\begin{prop}
There exist a constant $c > 0$ and a uniformly regular base transform $T = T(\eta,\epsilon)$ 
such that the matrix $\tilde H(\eta,\epsilon)=T^{-1}(\eta,\epsilon)H(\eta,\epsilon)T(\eta,\epsilon)$ satisfies 
$$
\tilde H(\eta,\epsilon) + \tilde H^*(\eta,\epsilon) \geq c\epsilon^2\eta I
$$
for sufficiently small $\epsilon$ and $\eta$.
\end{prop}
By introducing the parameterization $\eta=\rho \bar{\eta}$ and $\epsilon=\rho \bar{\epsilon}$, we know that it suffices to prove the result 
for all small $\rho>0$ and 
$$
(\bar{\eta},\bar{\epsilon})\in \Omega=\{\bar{\eta}=1,~~0<\bar{\epsilon}\leq 1\} \cup \{\bar{\epsilon}=1,~~0<\bar{\eta}\leq 1\} .
$$
Next we partition $\Omega$ into three parts, i.e.
$\Omega=\Omega_1\cup\Omega_2\cup\Omega_3$ where 
$$
\Omega_1 = \{\bar{\eta}=1,~~0<\bar{\epsilon}\leq c_1\},\quad \Omega_2=
\{\bar{\epsilon}=1,~~0< \bar{\eta}\leq c_2\} 
$$
with $c_1,c_2$ being sufficiently small positive constants. Moreover, the set $\Omega_3$ consists of the remaining points in $\Omega$. The set $\Omega=\Omega_1\cup\Omega_2\cup\Omega_3$ is illustrated in the figure below.
$$
\begin{tikzpicture}[scale=1]
\draw[->] (-0.5,0) -- (2.5,0) node[right] {$\bar{\epsilon}$};
\draw[->] (0,-0.5) -- (0,2.5) node[above] {$\bar{\eta}$};
\draw[thick] (2,0) -- (2,2);
\draw[thick] (0,2) -- (2,2);
\node[left] at (0,2) {$ 1 $};
\node[below] at (2,0) {$ 1 $};
\node[below left] at (0,0) {$0$};
\node at (0.05,2) {$($};
\node at (0.5,2) {$]$};
\node[below right] at (0,1.8) {$\Omega_1$};
\node[rotate=90] at (2,0.05) {$($};
\node[rotate=90] at (2,0.5) {$]$};
\node[above left] at (1.8,0) {$\Omega_2$};
\node[above right] at (1.9,1.9) {$\Omega_3$};
\end{tikzpicture}
$$
For these three parts, we prove the following three lemmas respectively.
\begin{lemma}\label{lemmaK}
For the set $~\Omega_1 = \{\bar{\eta}=1,~~0<\bar{\epsilon}\leq c_1\}$, we consider the matrix 
$$
H(\rho,\rho\bar{\epsilon})=\rho
\begin{pmatrix}
i\rho\bar{\epsilon}^2\tilde A_{11} &\ i\bar{\epsilon}\tilde A_{12} \\[1mm]
\ i\bar{\epsilon}\tilde A_{21} & \tilde\mJ    
\end{pmatrix}:=\rho K(\rho,\bar{\epsilon}).
$$    
There exist a constant $c > 0$ and a uniformly regular base transform $T = T(\rho,\bar{\epsilon})$ 
such that the matrix $\tilde K(\rho,\bar{\epsilon}) =T^{-1}(\rho,\bar{\epsilon})K(\rho,\bar{\epsilon})T(\rho,\bar{\epsilon})$ satisfies 
$$
\tilde K(\rho,\bar{\epsilon}) + \tilde K^*(\rho,\bar{\epsilon}) \geq c\rho^2\bar{\epsilon}^2 I
$$
for sufficiently small $\rho$ and $\bar{\epsilon}$.
\end{lemma}
\begin{lemma}\label{lemmaL}
For the set $\Omega_2=
\{\bar{\epsilon}=1,~~0< \bar{\eta}\leq c_2\}$, we consider the matrix 
$$
H(\rho\bar{\eta},\rho)=\rho
\begin{pmatrix}
i\rho \tilde A_{11} &\ i \tilde A_{12} \\[1mm]
\ i \tilde A_{21} &\bar{\eta} \tilde\mJ    
\end{pmatrix}:=\rho L(\rho,\bar{\eta}).
$$    
There exist a constant $c > 0$ and a uniformly regular base transform $T = T(\rho,\bar{\eta})$ 
such that the matrix $\tilde L(\rho,\bar{\eta}) =T^{-1}(\rho,\bar{\eta})L(\rho,\bar{\eta})T(\rho,\bar{\eta})$ satisfies 
$$
\tilde L(\rho,\bar{\eta}) + \tilde L^*(\rho,\bar{\eta}) \geq c\rho^2\bar{\eta} I
$$
for sufficiently small $\rho$ and $\bar{\eta}$.
\end{lemma}
For the compact set $\Omega_3$, it suffices to consider the matrix 
$$
H(\rho\bar{\eta},\rho\bar{\epsilon}) =\rho\begin{pmatrix}
i\rho\bar{\epsilon}^2\tilde A_{11} &\ i\bar{\epsilon}\tilde A_{12} \\[1mm]
\ i\bar{\epsilon}\tilde A_{21} &\bar{\eta} \tilde\mJ    
\end{pmatrix}
$$
for fixed $\bar{\epsilon}$ and $\bar{\eta}$.
For simplicity, we denote 
$$
G(\rho)=\begin{pmatrix}
i\rho \hat{A}_{11} &i\hat{A}_{12} \\[1mm]
\ i\hat{A}_{21} & \hat{\mJ}  
\end{pmatrix}:=\begin{pmatrix}
i\rho\bar{\epsilon}^2\tilde A_{11} &\ i\bar{\epsilon}\tilde A_{12} \\[1mm]
\ i\bar{\epsilon}\tilde A_{21} &\bar{\eta} \tilde\mJ    
\end{pmatrix}.
$$
Then we state
\begin{lemma}\label{lemmaG}
For fixed $\bar{\eta}>0$ and $\bar{\epsilon}>0$, there exist a constant $c > 0$ and a uniformly regular base transform $T = T(\rho)$ for $G(\rho)$  
such that the matrix $\tilde G(\rho) =T^{-1}(\rho)G(\rho)T(\rho)$ satisfies 
$$
\tilde G(\rho) + \tilde G^*(\rho) \geq c\rho^2 I
$$
for sufficiently small $\rho$.
\end{lemma}
Next we prove the three lemmas.
\begin{proof}[Proof of Lemma \ref{lemmaK}]
~\\
As the proof is long and technically involved, we begin with an overview of the strategy and a sketch of the main steps to guide the reader through the argument.

\begin{itemize}
    \item[(1)] The multiplicity of the zero eigenvalue for $K(0,0)$ equals to $5$. Accordingly, we decompose the whole space as 
    $\mathbb{C}^{14}=E_0\oplus E_1$ with $\text{dim}(E_0)=5$. By the perturbation theory, for sufficiently small $\rho$ and $\bar{\epsilon}$, we also have 
$$\mathbb{C}^{14}=E_0(\rho,\bar{\epsilon})\oplus E_1(\rho,\bar{\epsilon})$$ with $\text{dim}(E_0(\rho,\bar{\epsilon}))=5$. We can show that the real part of 
$K(\rho,\bar{\epsilon})|_{E_1}$ is larger than $cI$ for a constant $c$. 
\item[(2)] We observe that $K(\rho,\bar{\epsilon})|_{E_0}$ has a factor of $\bar{\epsilon}^2$. Namely, we write 
$$
K(\rho,\bar{\epsilon})|_{E_0} = V(\rho,\bar{\epsilon}) = \bar{\epsilon}^2 \hat{V}(\rho,\bar{\epsilon})
$$
with $\hat{V}(\rho,\bar{\epsilon})$ being regular w.r.t. $\rho$ and $\bar{\epsilon}$. 
The matrix $\hat{V}(0,0)$ has one simple zero eigenvalue and $4$ positive eigenvalues. Accordingly, we divide 
$$
E_0(\rho,\bar{\epsilon})=E_0^{(1)}(\rho,\bar{\epsilon})\oplus E_0^{(2)}(\rho,\bar{\epsilon})
$$
with $\text{dim}(E_0^{(1)}(\rho,\bar{\epsilon}))=1$. Moreover, we can prove that the real part of $\hat{V}(\rho,\bar{\epsilon})|_{E_0^{(2)}}$ is larger than $c I$ for a constant $c$.
\item[(3)] By computation, we find that 
$\hat{V}(\rho,\bar{\epsilon})|_{E_0^{(1)}}$ has a factor of $\rho^2$. Namely, we write 
$$
K(\rho,\bar{\epsilon})|_{E_0^{(1)}} =\bar{\epsilon}^2\hat{V}(\rho,\bar{\epsilon})|_{E_0^{(1)}} = \rho^2\bar{\epsilon}^2 \bar{V}(\rho,\bar{\epsilon})
$$ 
with $\bar{V}(\rho,\bar{\epsilon})$ being another regular matrix.
At last, we finish the proof by showing that $\bar{V}(0,0)>0$.
\end{itemize}

\textbf{Step 1:}
At the point $(\rho,\bar{\epsilon})=(0,0)$, the unperturbed matrix reads
    $$
    K(0,0)=\begin{pmatrix}
        0 & \\
          & \tilde \mJ
    \end{pmatrix}.
    $$
    Thus there exists a regular transform $T_1=T_1(\rho,\bar{\epsilon})$ such that 
    $$
    T_1^{-1}(\rho,\bar{\epsilon})K(\rho,\bar{\epsilon})T_1(\rho,\bar{\epsilon}) = \begin{pmatrix}
        V(\rho,\bar{\epsilon}) & \\[1mm]
          & W(\rho,\bar{\epsilon})
    \end{pmatrix}
    $$
    with $V(\rho,\bar{\epsilon}) \in \mathbb{C}^{5\times 5}$, $W(\rho,\bar{\epsilon})\in \mathbb{C}^{9\times 9}$ and
    $$
    T_1(\rho,\bar{\epsilon}) = 
    \begin{pmatrix}
        I  & X(\rho,\bar{\epsilon})\\[1mm]
        Y(\rho,\bar{\epsilon})  & I
    \end{pmatrix}.
    $$
    Clearly, we have 
    $$
    V(0,0)=0,\qquad W(0,0)=\tilde\mJ.
    $$
    Therefore, for sufficiently small $\rho$ and $\bar{\epsilon}$, we have
    $$
    W(\rho,\bar{\epsilon})+W^*(\rho,\bar{\epsilon})\geq cI.
    $$
    To compare this formula with the sketch of proof, we remark that the matrix $V(\rho,\bar{\epsilon})$ corresponds to the subspace $E_0(\rho,\bar{\epsilon})$ and $W(\rho,\bar{\epsilon})$ corresponds to $E_1(\rho,\bar{\epsilon})$.
    
    \textbf{Step 2:} Next we investigate the matrix $V(\rho,\bar{\epsilon}) \in \mathbb{C}^{5\times 5}$. At first, we state that the matrix can be written as 
    \begin{equation}\label{lemma3-V}
    V(\rho,\bar{\epsilon})=\bar{\epsilon}^2 \hat{V}(\rho,\bar{\epsilon})
    \end{equation} 
    with $\hat{V}(\rho,\bar{\epsilon})\in \mathbb{C}^{5\times 5}$ being regular with respect to $\rho$ and $\bar{\epsilon}$. Indeed, for fixed
    $\rho>0$, we consider the computation in Appendix with the varying parameter  $\bar{\epsilon}$. In this case, we have $$V(\rho,0)=0.$$ Moreover, the relation \eqref{Vprime} shows that
    $$
    \frac{\partial V}{\partial \bar{\epsilon}}(\rho,0)=0.
    $$ 
    These two relations imply \eqref{lemma3-V}.
     Moreover, the computations \eqref{Yprime} and \eqref{Vprimeprime} give
    \begin{align*}
        \frac{\partial Y}{\partial \bar{\epsilon}}(\rho,0) =-(C_{22}(\rho,0))^{-1}\frac{\partial C_{21}}{\partial \bar{\epsilon}}(\rho,0)=-i\tilde \mJ^{-1} \tilde{A}_{21}
    \end{align*}
    and
    \begin{align*}
    \frac{\partial^2 V}{\partial \bar{\epsilon}^2}(\rho,0)=&~\frac{\partial^2 C_{11}}{\partial \bar{\epsilon}^2}(\rho,0) +2\frac{\partial C_{12}}{\partial \bar{\epsilon}}(\rho,0) \frac{\partial Y}{\partial \bar{\epsilon}}(\rho,0) +C_{12}(\rho,0)\frac{\partial^2 Y}{\partial \bar{\epsilon}^2}(\rho,0) \\[2mm]
    =&~2i\rho\tilde{A}_{11} + 2  \tilde{A}_{12} \tilde \mJ^{-1} \tilde{A}_{21}.
    \end{align*}
    Substituting the expression \eqref{lemma3-V} into this relation, we obtain
    \begin{equation}\label{lemma3-hatV}
    \hat{V}(\rho,0)=i\rho\tilde{A}_{11} +   \tilde{A}_{12} \tilde \mJ^{-1} \tilde{A}_{21}.
    \end{equation}
    Since $\hat{V}(0,0)=  \tilde{A}_{12} \tilde \mJ^{-1} \tilde{A}_{21}$ has one simple zero eigenvalue and $4$ positive eigenvalues, there exists a regular transform $S(\rho,\bar{\epsilon})$ such that 
    \begin{equation}\label{lemma3-V0V1}
    S^{-1}(\rho,\bar{\epsilon})\hat{V}(\rho,\bar{\epsilon})S(\rho,\bar{\epsilon}) = \begin{pmatrix}
        V_0(\rho,\bar{\epsilon}) & \\[1mm]
          & V_1(\rho,\bar{\epsilon})
    \end{pmatrix}
    \end{equation}
with $V_0(0,0)=0$ and $V_1(0,0)$ being a $4\times 4$ symmetric positive definite matrix. To sum up, we have 
$$
\begin{pmatrix}
S^{-1}(\rho,\bar{\epsilon}) & \\[1mm]
        & I
\end{pmatrix}T_1^{-1}(\rho,\bar{\epsilon})K(\rho,\bar{\epsilon})T_1(\rho,\bar{\epsilon}) \begin{pmatrix}
S(\rho,\bar{\epsilon}) & \\[1mm]
        & I
\end{pmatrix}
= \begin{pmatrix}
\bar{\epsilon}^2V_0(\rho,\bar{\epsilon}) & & \\[1mm]
 & \bar{\epsilon}^2V_1(\rho,\bar{\epsilon})& \\[1mm]
      &  & W(\rho,\bar{\epsilon})
\end{pmatrix},
$$
where 
$$
V_1(\rho,\bar{\epsilon})+V_1^*(\rho,\bar{\epsilon})\geq cI,\qquad W(\rho,\bar{\epsilon})+W^*(\rho,\bar{\epsilon})\geq cI
$$ 
for sufficiently small $\rho$ and $\bar{\epsilon}$.

To compare with the sketch of proof, we remark that the matrix $V_0(\rho,\bar{\epsilon})$ corresponds to the invariant space $E_0^{(1)}(\rho,\bar{\epsilon})$ and $V_1(\rho,\bar{\epsilon})$ corresponds to $E_0^{(2)}(\rho,\bar{\epsilon})$.

\textbf{Step 3:} Now we investigate the $\rho$-dependence of the number $V_0(\rho,\bar{\epsilon})\in\mathbb{C}$. To this end, we consider the original matrix $K(\rho,\bar{\epsilon})$ 
for fixed $\bar{\epsilon}>0$ and $\rho=0$:
$$
K(0,\bar{\epsilon})=\begin{pmatrix}
0 &\ i\bar{\epsilon}\tilde A_{12} \\[1mm]
\ i\bar{\epsilon}\tilde A_{21} & \tilde\mJ
\end{pmatrix}.
$$
We observe that it has one simple zero eigenvalue 
and 13 nonzero eigenvalues. Comparing with the computation in Step 2, we have
$$
\bar{\epsilon}^2V_0(0,\bar{\epsilon})\equiv 0.
$$
Then we can write
$$
V_0(\rho,\bar{\epsilon})=\rho \tilde{V}_0(\rho,\bar{\epsilon})
$$
with $\tilde{V}_0(\rho,\bar{\epsilon})$ a regular function. Furthermore, the left and right eigenvectors for $K(0,\bar{\epsilon})$ associated with the zero eigenvalue are 
$$
l(0,\bar{\epsilon})=\begin{pmatrix}
    \alpha_1^T &
    0
\end{pmatrix},\qquad
r(0,\bar{\epsilon})=\begin{pmatrix}
    \alpha_1\\[1mm]
    0
\end{pmatrix}.
$$
Here $\alpha_1\in\mathbb{R}^{5}$ is the vector in $\ker(\tilde{A}_{21})$ with $|\alpha_1|=1$. 
Taking partial derivative $\partial/\partial \rho$ on the relation 
$$
0=l(\rho,\bar{\epsilon})\Big(K(\rho,\bar{\epsilon}) - \bar{\epsilon}^2 \rho \tilde{V}_0(\rho,\bar{\epsilon})I \Big)r(\rho,\bar{\epsilon})
$$
and evaluating at $(0,\bar{\epsilon})$ yield
$$
0=l(0,\bar{\epsilon})\Big(\frac{\partial K}{\partial \rho}(0,\bar{\epsilon}) - \bar{\epsilon}^2 \tilde{V}_0(0,\bar{\epsilon})I \Big)r(0,\bar{\epsilon}).
$$
Notice that 
$$
l(0,\bar{\epsilon}) \frac{\partial K}{\partial \rho}(0,\bar{\epsilon}) r(0,\bar{\epsilon})=i\bar{\epsilon}^2\alpha_1^T\tilde{A}_{11}\alpha_1=0.
$$
Then $\tilde{V}_0(0,\bar{\epsilon})=0$ for any $\bar{\epsilon}$ and hence
$$
V_0(\rho,\bar{\epsilon})=\rho \tilde{V}_0(\rho,\bar{\epsilon}) = \rho^2\bar{V}_0(\rho,\bar{\epsilon})
$$
with $\bar{V}_0(\rho,\bar{\epsilon})$ being a regular function.  

It remains to show that $\bar{V}_0(0,0)>0$. We recall 
\eqref{lemma3-V0V1} that  
\begin{equation}\label{lemma3-matrixS}
S^{-1}(\rho,0)\hat{V}(\rho,0)S(\rho,0) = \begin{pmatrix}
        V_0(\rho,0) & \\[1mm]
          & V_1(\rho,0)
    \end{pmatrix}=\begin{pmatrix}
        \rho^2\bar{V}_0(\rho,0) & \\[1mm]
          & V_1(\rho,0)
    \end{pmatrix}.
\end{equation}
In order to compute $\bar{V}_0(\rho,0)$, we need to introduce a block-matrix expression of $S(\rho,0)$.
According to \eqref{lemma3-hatV}, we take 
$$
S_0=\begin{pmatrix}
    \alpha_1 & \alpha_1^{\perp}
\end{pmatrix}
$$
with $\alpha_1\in\mathbb{R}^{5}$ being the vector in $\ker(\tilde{A}_{21})$ and $\alpha_1^{\perp}\in \mathbb{C}^{5\times 4}$ being the orthogonal complement of $\alpha_1$, 
and compute that 
$$
S^{-1}_0\hat{V}(\rho,0)S_0 = S^T_0\hat{V}(\rho,0)S_0 = \begin{pmatrix}
        0 & i\rho K_{12}\\[1mm]
        i\rho K_{21} & K_{22}(\rho)
\end{pmatrix}. 
$$
Here we denote
\begin{align*}
K_{12} = K_{21}^T 
=\alpha_1^T\tilde{A}_{11}\alpha_1^{\perp},
\quad
K_{22}(\rho)= i\rho\Big[(\alpha_1^{\perp})^T\tilde{A}_{11}\alpha_1^{\perp}\Big]+(\alpha_1^{\perp})^T(\tilde{A}_{12} \tilde \mJ^{-1} \tilde{A}_{21})\alpha_1^{\perp}.
\end{align*}
Since $K_{22}(0)$ is a symmetric positive definite matrix, there exists a regular transform 
$$
\bar{S}(\rho)=\begin{pmatrix}
        I  & X_1(\rho)\\[1mm]
        Y_1(\rho)  & I
    \end{pmatrix}
$$
such that 
$$
\bar{S}^{-1}(\rho)
\begin{pmatrix}
        0 & i\rho K_{12} \\[1mm]
        i\rho K_{21} & K_{22}(\rho)
\end{pmatrix}\bar{S}(\rho) = 
\begin{pmatrix}
        V_0(\rho,0) & \\[1mm]
          & V_1(\rho,0)
    \end{pmatrix}=
\begin{pmatrix}
        \rho^2\bar{V}_0(\rho,0) & \\[1mm]
          & V_1(\rho,0)
    \end{pmatrix}.
$$
Namely, the transform matrix in \eqref{lemma3-matrixS} has the form
$$
S(\rho,0)=S_0\bar{S}(\rho).
$$
Then we exploit the computation in Appendix. By \eqref{Vprimeprime} and \eqref{Yprime}, we have 
$$
\frac{d Y_1}{d\rho}(0)=-iK_{22}^{-1}(0)K_{21} 
$$
and 
$$
\frac{\partial^2V_0}{\partial\rho^2}(0,0)=2iK_{12} Y_1'(0)=2K_{12} K_{22}^{-1}(0)K_{21} >0.
$$
Therefore we conclude 
$$
\bar{V}_0(0,0)=\frac12\frac{\partial^2V_0}{\partial\rho^2}(0,0)>0
$$
and complete the proof.
\end{proof}

\begin{proof}[Proof of Lemma \ref{lemmaL}]
~\\
We begin with an overview of the strategy and a sketch of the main steps.
\begin{itemize}
    \item[(1)] The multiplicity of the zero eigenvalue for $L(0,0)$ equals to $6$. Accordingly, we decompose the whole space as 
    $\mathbb{C}^{14}=E_0\oplus E_1$ with $\text{dim}(E_0)=6$. By the perturbation theory, for sufficiently small $\rho$ and $\bar{\eta}$, we also have 
$$\mathbb{C}^{14}=E_0(\rho,\bar{\eta})\oplus E_1(\rho,\bar{\eta})$$ with $\text{dim}(E_0(\rho,\bar{\eta}))=6$.
\item[(2)] We observe that $L(\rho,\bar{\eta})|_{E_0}$ has a factor of $\bar{\eta}$. Namely, we write 
$$
L(\rho,\bar{\eta})|_{E_0} = V(\rho,\bar{\eta}) = \bar{\eta} \bar{V}(\rho,\bar{\eta})
$$
with $\bar{V}(\rho,\bar{\eta})$ being regular w.r.t. $\rho$ and $\bar{\eta}$. 
The matrix $\bar{V}(0,0)$ has one simple zero eigenvalue and $5$ positive eigenvalues.
Then accordingly we decompose 
$$E_0(\rho,\bar{\eta})=E_0^{(1)}(\rho,\bar{\eta})\oplus E_0^{(2)}(\rho,\bar{\eta})$$
with $\text{dim}(E_0^{(1)}(\rho,\bar{\eta}))=1$.
\item[(3)] We prove the property 
$$
\dfrac{\partial L}{\partial\bar{\eta}}\Big|_{E_0^{(2)}\cup E_1}>0
$$
and show that the real part of $
L(\rho,\bar{\eta})|_{E_0^{(2)}\cup E_1}$ is larger than $c\bar{\eta}I$ for a constant $c$.
\item[(4)] By using the specific expression of basis function in $E_0^{(1)}$, we can prove that $L(\rho,\bar{\eta})|_{E_0^{(1)}}$ has a factor of $\rho^2$ and satisfies
$$L(\rho,\bar{\eta})|_{E_0^{(1)}}\geq c\rho^2\bar{\eta}.$$
\end{itemize}

\textbf{Step 1:}
Consider the matrix  
    $$
    L(\rho,0)=i\begin{pmatrix}
        \rho\tilde{A}_{11} & \tilde{A}_{12}\\[1mm]
        \tilde{A}_{21} & 0
    \end{pmatrix}
    $$
with parameters $\rho>0$ and $\bar{\eta}=0$.
There exists a regular orthogonal transform $\mathcal{O}(\rho)$ such that 
    \begin{equation}
    \mathcal{O}^{-1}(\rho)L(\rho,0)\mathcal{O}(\rho) = \begin{pmatrix}
        V_0(\rho) & & & \\[1mm]
        & V_1(\rho) & &\\[1mm]
        & & \ddots &\\[1mm]
        &  & & V_p(\rho)
    \end{pmatrix}.
    \end{equation}
    where $V_0(0)=0$ and each $V_k(\rho)\in \C^{n_k\times n_k}$ satisfies 
    $$
V_k(0)=i\Lambda_k=\text{diag}(\lambda_k,\dots,\lambda_k),\qquad 1\leq k\leq p.
    $$ 
Here $p$ represents the number of nonzero eigenvalues and $\lambda_k$ is the $k$-th nonzero eigenvalue of $L(0,0)$ with the multiplicity $n_k$. 
    
For further derivation, we  introduce the expression of $\mathcal{O}(\rho)$ as follows:
$$
\mathcal{O}(\rho)=\begin{pmatrix}
r_1(\rho),r_2(\rho),\dots,r_6(\rho),r_7(\rho),\dots,r_{14}(\rho)
\end{pmatrix}.
$$
Note that the dimension of kernel for $L(\rho,0)$ is $6$ for any $\rho>0$. Indeed, the basis of the kernel is given by 
\begin{equation}\label{lemma4-r1}
r_1(\rho)=\frac{1}{N_\rho}
\begin{pmatrix}
\alpha_1 \\[1mm]
\rho\,\beta_1
\end{pmatrix},\qquad (r_2,r_3,\dots,r_6)(\rho)\equiv\begin{pmatrix}
    0 & 0 & \dots & 0\\[1mm]
    \beta_2 & \beta_3 &\dots & \beta_6
\end{pmatrix}.
\end{equation}
Here $\alpha_1\in\ker(\tilde A_{21})$, $\beta_1$ is determined by
\begin{equation}\label{lemma4-eqal}
\tilde A_{12}\beta_1=-\tilde A_{11}\alpha_1,
\end{equation} $\beta_2,\beta_3,\dots,\beta_6\in \mathbb{R}^{9}$
form a basis of $\ker(\tilde A_{12})$, and $N_\rho$ is the normalization factor.
We show that the algebraic equation \eqref{lemma4-eqal} has a solution $\beta_1$.
Due to $\alpha_1^T\tilde A_{12}=0$ and $\alpha_1^T\tilde A_{11}\alpha_1=0$, we know that the rank of the augmented matrix $(\tilde A_{12}, \tilde A_{11}\alpha)$ equals to $4$ and thereby the equation \eqref{lemma4-eqal} has a solution $\beta_1$. Clearly, we can obtain $\beta_1$ such that it is orthogonal to $(\beta_2,\dots,\beta_6)$. Moreover, since $\alpha_1\in\ker(\tilde A_{21})$ and $\ker(\tilde A_{21})\cap \ker(\tilde A_{11}) = \{0\}$, we know that $\beta_1\neq 0$.

When $\rho=0$, the vectors $(r_7,r_8,\dots,r_{14})(0)$ correspond to the eigenvectors of $L(0,0)$ associated with nonzero eigenvalues. By a direct computation, we have
$$
r_k(0) = \begin{pmatrix}
    \alpha_k\\[1mm]
    \beta_k
\end{pmatrix}\quad \text{with}~~|\beta_k|\geq c ~~\text{for}~~ 7\leq k\leq 14.
$$
Spectral separation of $V_k(\rho)$ implies the existence of a smooth family $T_1(\rho,\bar{\eta})$ such that 
\begin{equation}\label{lemma4-eq1}
    T_1^{-1}(\rho,\bar{\eta})L(\rho,\bar{\eta})T_1(\rho,\bar{\eta}) = \begin{pmatrix}
        V_0(\rho,\bar{\eta}) & & & \\[1mm]
        & V_1(\rho,\bar{\eta}) & &\\[1mm]
        & & \ddots &\\[1mm]
        &  & & V_p(\rho,\bar{\eta})
    \end{pmatrix}.
    \end{equation}
When $\bar{\eta}=0$, we have $T_1(\rho,0)=\mathcal{O}(\rho)$. 

To compare this formula with the sketch of proof, we remark that the matrix $V_0(\rho,\bar{\eta})$ corresponds to the subspace $E_0(\rho,\bar{\eta})$ and $\text{diag}(V_1(\rho,\bar{\eta}),\dots,V_p(\rho,\bar{\eta}))$ corresponds to $E_1(\rho,\bar{\eta})$.


\textbf{Step 2:} From the computation in Step 1, we know that the dimension of the kernel of $L(\rho,0)$ equals to 6 for any $\rho>0$. This means
$$
V_0(\rho,0)=0
$$
for any $\rho>0$.
Then we write 
\begin{align}
V_0(\rho,\bar{\eta})=\bar{\eta}\bar{V}_0(\rho,\bar{\eta}).\label{lemma4-ReV}
\end{align}
with $\bar{V}_0\in \mathbb{C}^{6\times 6}$ being regular w.r.t. $(\rho,\bar{\eta})$. 
By a direct computation, we have
\begin{align*}
\bar{V}_0(0,0)=&~
\frac{\partial V_0}{\partial \bar{\eta}}(0,0)\\[1mm]
    =&~(I_6,~0)\frac{\partial }{\partial \bar{\eta}}\Big(T_1^{-1}(\rho,\bar{\eta})L(\rho,\bar{\eta})T_1(\rho,\bar{\eta})\Big)\begin{pmatrix}
        I_6\\[1mm]
        0
    \end{pmatrix}\bigg|_{(\rho,\bar{\eta})=(0,0)}\\[2mm]
    =&~(I_6,~0)\frac{\partial (T_1^{-1})}{\partial \bar{\eta}}(0,0) L(0,0)T_1(0,0) \begin{pmatrix}
        I_6\\[1mm]
        0
    \end{pmatrix}+(I_6,~0)T_1^{-1}(0,0) L(0,0)\frac{\partial T_1}{\partial \bar{\eta}}(0,0) \begin{pmatrix}
        I_6\\[1mm]
        0
    \end{pmatrix} \\[2mm]
    &+(I_6,~0)T_1^{-1}(0,0)\frac{\partial L}{\partial \bar{\eta}}(0,0)T_1(0,0) \begin{pmatrix}
        I_6\\[1mm]
        0
    \end{pmatrix} \\[2mm]
    =&~(r_1(0),\dots,r_6(0))^T\begin{pmatrix}
        0 & \\
          & \tilde \mJ
    \end{pmatrix}(r_1(0)\dots,r_6(0)).
\end{align*}
According to the expression of $r_1(0),\dots,r_6(0)$, we observe that $\bar{V}(0,0)$ has one simple zero eigenvalue and 5 positive eigenvalues. Thanks to this, there exists a regular transform $S(\rho,\bar{\eta})$ such that $S(0,0)=I$ and 
$$
    S^{-1}(\rho,\bar{\eta})\bar{V}_0(\rho,\bar{\eta})S(\rho,\bar{\eta}) = \begin{pmatrix}
        d(\rho,\bar{\eta}) & \\[1mm]
          & D(\rho,\bar{\eta})
    \end{pmatrix}.
    $$
To conclude, we have 
\begin{equation}\label{Lemma3-T-decom}
    T^{-1}(\rho,\bar{\eta})L(\rho,\bar{\eta})T(\rho,\bar{\eta}) = \begin{pmatrix}
        \bar{\eta}d(\rho,\bar{\eta}) & & & & \\[1mm]
        & \bar{\eta}D(\rho,\bar{\eta}) & & & \\[1mm]
        & & V_1(\rho,\bar{\eta}) & & \\[1mm]
        & & & \ddots &\\[1mm]
        & & & & V_p(\rho,\bar{\eta})
    \end{pmatrix}.
    \end{equation}
with 
$$
T(\rho,\bar{\eta}) = T_1(\rho,\bar{\eta}) \begin{pmatrix}
S(\rho,\bar{\eta}) & \\[1mm]
        & I
\end{pmatrix}.
$$
To compare this formula with the sketch, we remark that $d(\rho,\bar{\eta})$ corresponds to the subspace $E_0^{(1)}(\rho,\bar{\eta})$ and $D(\rho,\bar{\eta}))$ corresponds to $E_0^{(2)}(\rho,\bar{\eta})$.

\textbf{Step 3:}
For fixed $\rho>0$ and $\bar{\eta}=0$, we know that the matrix $L(\rho,0)$ is skew-Hermitian. Thus 
$$
V_k(\rho,0)+V_k^*(\rho,0)=0
$$
and it follows that 
$$
V_k(\rho,\bar{\eta})+V_k^*(\rho,\bar{\eta}) = \bar{\eta}\bar{V}_k(\rho,\bar{\eta}).
$$

Denote 
$$
J_0=(0,I_{5\times 5},0,\dots,0)^T,\qquad 
J_k=(0,\dots,0,I_{n_k},0,\dots,0)^T,\quad k\geq 1.
$$ 
We compute 
\begin{align*}
    &~J_k^T\frac{\partial }{\partial \bar{\eta}}\Big(T^{-1}(\rho,\bar{\eta})L(\rho,\bar{\eta})T(\rho,\bar{\eta})-\lambda_k I\Big)J_k\bigg|_{(\rho,\bar{\eta})=(0,0)}\\[2mm]
    =&~J_k^T\frac{\partial (T^{-1})}{\partial \bar{\eta}}(0,0) [L(0,0)-\lambda_kI] T(0,0) J_k+J_k^TT^{-1}(0,0) [L(0,0)-\lambda_kI]\frac{\partial T}{\partial \bar{\eta}}(0,0) J_k \\[2mm]
    &+J_k^TT^{-1}(0,0)\frac{\partial L}{\partial \bar{\eta}}(0,0)T(0,0) J_k \\[2mm]
    =&~J_k^T\mathcal{O}^{T}(0)\begin{pmatrix}
        0 & \\
          & \tilde \mJ
    \end{pmatrix}\mathcal{O}(0) J_k.
\end{align*}
The above matrix is positive definite for every \(k\geq0\). For \(k=0\), this follows from the linear independence of \(\beta_2,\ldots,\beta_6\) and \(\widetilde{\mathcal J}>0\). For \(k\geq1\), any vector in the corresponding nonzero eigenspace of \(L(0,0)\) with vanishing dissipative component must vanish, so the restriction of \(\operatorname{diag}(0,\widetilde{\mathcal J})\) to this eigenspace is positive definite.
By \eqref{Lemma3-T-decom}, the computation implies that 
$$
D(0,0)=\frac{\partial (\bar{\eta}D(\rho,\bar{\eta}))}{\partial \bar{\eta}}\Big|_{(\rho,\bar{\eta})=(0,0)}\geq cI,\quad \bar{V}_k(0,0)= \frac{\partial (V_k(\rho,\bar{\eta})+V_k^*(\rho,\bar{\eta}))}{\partial \bar{\eta}}\Big|_{(\rho,\bar{\eta})=(0,0)}\geq cI.
$$
Therefore, for sufficiently small $\rho$ and $\bar{\eta}$, we have 
$$
D(\rho,\bar{\eta})+D^*(\rho,\bar{\eta})\geq cI,\qquad V_k(\rho,\bar{\eta})+V_k^*(\rho,\bar{\eta}) \geq c\bar{\eta}I.
$$

\textbf{Step 4:} It remains to compute $d(\rho,\bar{\eta})$. For $\rho=0$ and $\bar{\eta}>0$, we consider the matrix $$
L(0,\bar{\eta})=\begin{pmatrix}
    0 & i\tilde{A}_{12}\\[1mm]
    i\tilde{A}_{21} & 0
\end{pmatrix}+\bar{\eta}
\begin{pmatrix}
    0 & 0\\[1mm]
    0 & \tilde\mJ
\end{pmatrix}
$$
For any $\bar{\eta}>0$, $r_1(0)$ is the eigenvector of this matrix associated with zero eigenvalue. We thus have
$$
d(0,\bar{\eta})=0
$$
for any $\bar{\eta}>0$. This means that there is a regular function $\bar{d}(\rho,\bar{\eta})$ such that
$$
d(\rho,\bar{\eta})=\rho \bar{d}(\rho,\bar{\eta}).
$$
Denote $r(\rho,\bar{\eta})$ and $l(\rho,\bar{\eta})$ as the first column of $T(\rho,\bar{\eta})$ and the first row of $T^{-1}(\rho,\bar{\eta})$.
Then we take partial derivative $\partial/\partial \rho$ on the relation 
$$
0=l(\rho,\bar{\eta})
\Big( L(\rho,\bar{\eta}) - \rho\bar{\eta} \bar{d}(\rho,\bar{\eta})I \Big)
r(\rho,\bar{\eta})
$$
at $(0,\bar{\eta})$ and obtain
$$
0=l(0,\bar{\eta})
\Big( \frac{\partial L}{\partial \rho}(0,\bar{\eta}) - \bar{\eta} \bar{d}(0,\bar{\eta})I \Big)
r(0,\bar{\eta}).
$$
Since $r(0,\bar{\eta})=r_1(0)$, $l(0,\bar{\eta})=r_1^T(0)$ and 
$$
r_1^T(0) \frac{\partial L}{\partial \rho}(0,\bar{\eta})r_1(0) = i\alpha_1^T\tilde{A}_{11} \alpha_1 = 0, 
$$
we have $\bar{\eta} \bar{d}(0,\bar{\eta})=0$ and thereby
$$
d(\rho,\bar{\eta})=\rho^2\hat{d}(\rho,\bar{\eta})
$$
with $\hat{d}(\rho,\bar{\eta})$ a regular function.


Let
$R(\rho)=(r_1(\rho),r_2(\rho),\ldots,r_6(\rho))$ be the
orthogonal basis of $\ker L(\rho,0)$ constructed in Step~1.
Since $L(\rho,0)R(\rho)=0$ and
$R^*(\rho)L(\rho,0)=0$, differentiation of the
six-dimensional block of \eqref{lemma4-eq1} at $\bar\eta=0$ gives
\begin{align*}
\bar V_0(\rho,0)
 &=R^*(\rho)
\begin{pmatrix}0&0\\0&\widetilde \mJ\end{pmatrix}
 R(\rho) = \Big(\frac{\rho}{N_{\rho}}\beta_1,\beta_2,\dots,\beta_6 \Big)^*\widetilde \mJ\Big(\frac{\rho}{N_{\rho}}\beta_1,\beta_2,\dots,\beta_6 \Big).
\end{align*}
Since $\widetilde \mJ > 0$ and $\beta_1,\dots,\beta_6$ are orthogonal to each other, we have 
\begin{align*}
\bar V_0(\rho,0)
 \geq c_0 \rho^2 I
\end{align*}
for sufficiently small $\rho$. In particular, the
eigenvalue $d(\rho,0)$ satisfies $d(\rho,0)\ge c_0\rho^2$.
Since $d(\rho,\bar\eta)=\rho^2\widehat d(\rho,\bar\eta)$
and $\widehat d$ is continuous, we obtain
$\operatorname{Re}\widehat d(\rho,\bar\eta)\ge c>0$
for sufficiently small $\rho$ and $\bar\eta$.
This completes the proof.

\end{proof}

\begin{proof}[Proof of Lemma \ref{lemmaG}] 
Consider the transform
$$
T_0 = \begin{pmatrix}
\alpha_1 & \alpha_1^{\perp} & 0\\[1mm]
 0 & 0 & I
\end{pmatrix}
$$
with 
$\alpha_1\in\ker{\hat{A}_{21}}$ and $\alpha_1^{\perp}$ being the orthogonal complement of $\alpha_1$ in $\mathbb{R}^5$. Then we compute
\begin{align*}
T_0^TG(\rho)T_0=&~\begin{pmatrix}
\alpha_1^T  & 0~\\[2mm]
(\alpha_1^{\perp})^T  & 0~\\[2mm]
0 & I~
\end{pmatrix}\left[\begin{pmatrix}
0 &i\hat{A}_{12} \\[2mm]
\ i\hat{A}_{21} & \hat{\mJ}  
\end{pmatrix}+\begin{pmatrix}
i\rho\hat{A}_{11} & 0 \\[2mm]
0 & 0
\end{pmatrix}\right]
\begin{pmatrix}
\alpha_1 & \alpha_1^{\perp} & 0\\[2mm]
 0 & 0 & I
\end{pmatrix}\\[2mm]
=&~
\begin{pmatrix}
~0 & 0 & 0 \\[2mm]
~0 & 0 & i(\alpha_1^{\perp})^T\hat{A}_{12} \\[2mm]
~0 & i\hat{A}_{21}\alpha_1^{\perp} & \hat{\mJ}  
\end{pmatrix}+\rho\begin{pmatrix}
0 & i\alpha_1^T\hat{A}_{11}\alpha_1^{\perp} & 0 \\[2mm]
i(\alpha_1^{\perp})^T\hat{A}_{11}\alpha_1 & i(\alpha_1^{\perp})^T\hat{A}_{11}\alpha_1^{\perp} & 0 \\[2mm]
0 & 0 & 0
\end{pmatrix}.
\end{align*}
Notice that the matrix $T_0^TG(0)T_0$ has one simple zero eigenvalue and 13 nonzero eigenvalues. Then there exists a regular transform $T_1(\rho)$ such that 
\begin{align}\label{lemma5-eq1}
T_1^{-1}(\rho)T_0^TG(\rho)T_0T_1(\rho)=\begin{pmatrix}
    \lambda(\rho) & \\[1mm]
     & \Lambda(\rho)
\end{pmatrix}.    
\end{align}
with 
$$
\lambda(0)=0,\qquad \Lambda(0)=
\begin{pmatrix}
0 & iK^T \\[1mm]
iK & \hat{\mJ}
\end{pmatrix},\qquad     
T_1(\rho)=
\begin{pmatrix}
        1 & X(\rho)\\[1mm]
        Y(\rho) & I
\end{pmatrix}.
$$
Here $K=\hat{A}_{21}\alpha_1^{\perp}\in \mathbb{R}^{9\times 4}$ has full-column rank, $X(\rho)\in \mathbb{C}^{1\times 13}$ and $Y(\rho)\in \mathbb{C}^{13\times 1}$.

Since any eigenvector of $\begin{pmatrix}
0 & iK^T \\
iK & 0
\end{pmatrix}$ is not included in the kernel of 
$\begin{pmatrix}
0 & 0 \\
0 & \hat{\mJ}
\end{pmatrix}$, we know that the eigenvalues of $\Lambda(0)$ has positive real parts and there exists a transform $S$ such that 
\begin{equation*} 
S^{-1}\Lambda(0)S + \Big(S^{-1}\Lambda(0)S\Big)^* \geq c I.
\end{equation*}
Then for sufficiently small $\rho$, we have
\begin{equation}\label{lemma5-eq2}
\tilde{\Lambda}(\rho)+\tilde{\Lambda}^*(\rho)\geq c I,\qquad \tilde{\Lambda}(\rho)=S^{-1}\Lambda(\rho)S.
\end{equation}
On the other hand, by the computation in Appendix, we have 
$$
\lambda'(0) = 0, \qquad Y'(0)=-\Lambda^{-1}(0)\begin{pmatrix}
i(\alpha_1^{\perp})^T\hat{A}_{11}\alpha_1\\[1mm]
0
\end{pmatrix}
,\qquad
\lambda''(0) = 2
\begin{pmatrix}
i\alpha_1^T\hat{A}_{11}\alpha_1^{\perp} & 
0
\end{pmatrix}Y'(0).
$$
The inverse matrix of $\Lambda(0)$ reads 
$$
\Lambda^{-1}(0) = \begin{pmatrix}
(K^T\hat{\mJ}^{-1}K)^{-1} & -i(K^T\hat{\mJ}^{-1}K)^{-1}K^T\hat{\mJ}^{-1} \\[2mm]
-i\hat{\mJ}^{-1}K(K^T\hat{\mJ}^{-1}K)^{-1} & \hat{\mJ}^{-1}\left[I-K(K^T\hat{\mJ}^{-1}K)^{-1}K^T\hat{\mJ}^{-1}\right]
\end{pmatrix}.
$$
Thus we compute 
$$
\lambda''(0)=2\left(\alpha_1^T\hat{A}_{11}\alpha_1^{\perp}\right)(K^T\hat{\mJ}^{-1}K)^{-1}\left((\alpha_1^{\perp})^T\hat{A}_{11}\alpha_1\right) >0,
$$
which implies 
\begin{equation}\label{lemma5-eq3}
    \operatorname{Re} \lambda(\rho)\geq c\rho^2\qquad \text{for sufficiently small}~~\rho.
\end{equation}
Combining \eqref{lemma5-eq1}, \eqref{lemma5-eq2} and \eqref{lemma5-eq3}, we can introduce 
$$
\tilde{G}(\rho)=T^{-1}(\rho)G(\rho)T(\rho)\quad\text{with}\quad T=T_0T_1(\rho)\begin{pmatrix}
    1 & 0 \\[1mm]
    0 & S
\end{pmatrix}
$$
and obtain 
$$
\tilde{G}^*(\rho)+\tilde{G}(\rho)\geq c\rho^2.
$$
This completes the proof of Lemma \ref{lemmaG}.
\end{proof}

\section*{\normalsize{Acknowledgments}} 
Heinrich Freist\"uhler was supported by grant 
FR 822/11 through DFG Priority Research Programme 2410
Hyperbolic Balance Laws. Yizhou Zhou is funded by Alexander von Humboldt Foundation (Humboldt Research Fellowship Programme for Postdocs).

\par\bigskip

\begin{appendix}
\section{A computation on block diagonalisation}
In this appendix, we consider the computation of the following transform:
$$
T(\eta)^{-1}C(\eta)T(\eta)=\begin{pmatrix}
                                 V(\eta)&0\\[1mm]
                                 0&W(\eta) 
                                 \end{pmatrix}
$$
with  
$$
C(\eta)=
\begin{pmatrix}
            C_{11}(\eta)&C_{12}(\eta)\\[1mm]
            C_{21}(\eta)&C_{22}(\eta)
            \end{pmatrix}, \qquad T(\eta)=\begin{pmatrix}
        I&X(\eta)\\[1mm]
        Y(\eta)&I  
        \end{pmatrix}.
$$
This is equivalent to the four matrix equations
\begin{eqnarray}
C_{11}+C_{12}Y&=&V\label{V}\\[1mm]
C_{11}X+C_{12}&=&XW\label{XW}\\[1mm]
C_{21}+C_{22}Y&=&YV\label{YV}\\[1mm]
C_{21}X+C_{22}&=&W\label{W},
\end{eqnarray}
which combine to   
\begin{eqnarray}
C_{11}X+C_{12}&=&X(C_{21}X+C_{22})\label{X}\\[1mm]
Y(C_{11}+C_{12}Y)&=&C_{21}+C_{22}Y.\label{Y}
\end{eqnarray}
Assuming that the matrix $C_{22}(0)$ is invertible, $C_{11}(0)=0$ and \(C_{12}(0)=0,~C_{21}(0)=0\), we know that the above equations uniquely determine $X(\eta)$ and $Y(\eta)$, as unique smooth functions for small values of $\eta$, with 
$$X(0)=0,\qquad Y(0)=0.$$ 
From \eqref{V} and \eqref{W}, we have
$$
V(0)=C_{11}(0)=0,\qquad 
W(0)=C_{22}(0).
$$
Differentiating equation \eqref{V} yields 
\be\label{Vprime}
V'=C_{11}'+C_{12}'Y+C_{12}Y',
\ee
Differentiating equations \eqref{Y} and \eqref{Vprime} gives the relations
\begin{align}
Y'(C_{11}+C_{12}Y)+Y(C_{11}+C_{12}Y)'=&~C_{21}'+C_{22}'Y+C_{22}Y',\nonumber\\[1mm]
V''=&~C_{11}''+2C_{12}'Y'+C_{12}Y'',\label{Vprimeprime}
\end{align}
which show that 
\be\label{Yprime}
Y'(0)=-(C_{22}(0))^{-1}C_{21}'(0).
\ee
\end{appendix}


\begin{thebibliography}{22}

\bibitem{B} G. Boillat: {\it Sym\'etrisation des syst\`emes d'\'equations aux d\'eriv\'ees 
partielles avec densit\'e d'\'energie convexe et contraintes.} Comptes Rendus, S\'er.\ 1, {\bf 295} (1982), 551--554. 

\bibitem{By} S. Boyaval: {\it Viscoelastic flows of Maxwell fluids with conservation laws.}
ESAIM: Mathematical Modelling and Numerical Analysis {\bf 55} (2021), 807--831

\bibitem{CKX} T. Crin-Barat, S. Kawashima, J. Xu:
{\it The Cattaneo-Christov approximation of Fourier heat-conductive compressible fluids,}
Comm.\ Math.\ Phys. {\bf 406} (2025), article number 237.

\bibitem{HR3} 
Y. Hu, R. Racke: {\it Global existence versus blow-up for multidimensional hyperbolized compressible Navier-Stokes equations,} SIAM J. Math. Anal. {\bf 55} (2023), no.~5, 4788--4815.

\bibitem{HR2} 
Y. Hu, R. Racke: {\it Hyperbolic compressible Navier-Stokes equations,} J. Differential Equations {\bf 269} (2020), no.~4, 3196--3220.

\bibitem{HR1} 
Y. Hu, R. Racke, {\it Compressible Navier-Stokes equations with revised Maxwell's law,} J. Math. Fluid Mech. {\bf 19} (2017), no.~1, 77--90.

\bibitem{F22} H. Freist\"uhler: {\it Time-asymptotic stability for first-order 
symmetric hyperbolic systems of balance laws in dissipative compressible fluid dynamics}.
Quart.\ Appl.\ Math.\ {\bf 80} (2022), 597--606.

\bibitem{FS21} H. Freist\"uhler, M. Sroczinski:
{\it A class of uniformly dissipative symmetric hyperbolic-hyperbolic systems}.
J.\ Differ.\ Equ.\ {\bf 288} (2021), 40--61.

\bibitem{FL} 
K. O. Friedrichs, P.  Lax: {\it Systems of conservation equations with a convex extension,} Proc.\ Nat.\ Acad.\ Sci.\ USA {\bf 68} (1971), 1686--1688.   

\bibitem{G} 
S. K.   Godunov: {\it An interesting class of quasilinear systems.} 
Dokl.\ Akad.\ Nauk SSSR 139 (1961), 525--523. 

\bibitem{Kato-book} T. Kato:
{\it Perturbation theory for linear operators,} {Springer-Verlag, Berlin},
(1995), xxii+619,

\bibitem{K83} S. Kawashima: {\it Large-time behaviour of solutions to hyperbolic-parabolic systems of
conservation laws and applications.} Proc.\ Roy.\ Soc.\ Edinburgh Sect.\ A {\bf 106} (1987), 169--194.  


\bibitem{MR} I. M\"uller, T. Ruggeri: 
{\it Rational Extended Thermodynamics. Second edition,}  
Springer Tracts in Natural Philosophy, 
37. Springer-Verlag, New York, 1998.

\bibitem{PZ2} Y.-J. Peng and L. Zhao, {\it Global convergence rates from relaxed Euler equations to Navier-Stokes equations with Oldroyd-type constitutive laws,} Nonlinearity {\bf 37} (2024), no.~9, Paper No. 095032, 26 pp.

\bibitem{PZ1} Y. J. Peng, L. Zhao, {\it Global convergence to compressible full Navier-Stokes equations
 by approximation with Oldroyd-type constitutive laws,} J. Math. Fluid Mech., 24 (2022), 17.


\bibitem{R83} 
T. Ruggeri: {\it Symmetric-hyperbolic system of conservative equations for a viscous 
heat conducting fluid.} Acta Mech.\ {\bf 47} (1983), 167--183.

\bibitem{RS}
T. Ruggeri, M. Sugiyama: {\it Classical and relativistic rational extended thermodynamics of gases,}
Springer, Cham: 2021.  

\bibitem{SK85} Y. Shizuta, S. Kawashima: {\it Systems of equations of hyperbolic-parabolic type with applications to the discrete Boltzmann equation.} Hokkaido Math. J. {\bf 14} (1985), 249-275.
 
\bibitem{UKS} T. Umeda, S. Kawashima, Y. Shizuta:
{\it On the decay of solutions to the
linearized equations of electromagnetofluid dynamics}.
Japan J. Appl. Math. {\bf 1} (1984),
435--457.

\bibitem{Y} W.-A. Yong: \textit{Newtonian limit of Maxwell fluid flows.} Arch.\ Ration.\ Mech.\ Anal.\ {\bf 214} (2014), 913--922. 



\end{thebibliography}
\end{document}